\documentclass[a4paper]{article}

\usepackage{amsmath,amscd,amssymb,amsthm,amssymb}

\usepackage{authblk}

\usepackage{cite}
\usepackage{bibgerm}

\usepackage[colorlinks]{hyperref}
\hypersetup{
  citecolor={cyan}
}

\usepackage{adjustbox}

\usepackage[super]{nth}

\usepackage[UKenglish]{babel}

\usepackage{tikz-cd}

\usepackage{fullpage}

\usepackage{enumitem}

\newtheorem{thm}{Theorem}[section]
\newtheorem{prop}[thm]{Proposition}
\newtheorem{lem}[thm]{Lemma}
\newtheorem{cor}[thm]{Corollary}
\theoremstyle{definition}
\newtheorem{rem}[thm]{Remark}
\newtheorem{dfn}[thm]{Definition}
\newtheorem{exmpl}[thm]{Example}
\newtheorem{exmpls}[thm]{Examples}

\newcommand{\SOg}{\textrm{\textbf{SO}}}

\newcommand{\NN}{\mathbb N}

\newcommand{\ZZ}{\mathbb Z}
\newcommand{\RR}{\mathbb R}

\newcommand{\id}{\textrm{id}}

\newcommand{\Sq}{\textrm{Sq}}

\newcommand{\biwe}{\mathord{\adjustbox{valign=B,totalheight=.6\baselineskip}{$\bigwedge$}\kern-0.1em}}

\def\haken{\mathbin{\hbox to 5pt{%
                 \vrule height0.4pt width5pt depth0pt
                 \kern-.4pt
                 \vrule height6pt width0.4pt depth0pt\hss}}}

\usepackage{authblk}

\title{A counting invariant for maps into spheres and for zero loci of sections
of vector bundles}
\author{Panagiotis Konstantis}

\AtEndDocument{%
  \medskip
  \noindent
    Panagiotis Konstantis,
    Department Mathematik/Informatik, Abteilung Mathematik,
    Universität zu Köln, Weyertal 86-90, 50931 Köln, Germary,
    \textit{E-mail address}: \texttt{pako@math.uni-koeln.de}
}

\begin{document}
\maketitle
\begin{abstract}
The set of unrestricted homotopy classes $[M,S^n]$ where $M$ is a closed and connected spin
  $(n+1)$-manifold is called the $n$-th cohomotopy group $\pi^n(M)$ of $M$. Moreover it is known
  that $\pi^n(M) = H^n(M;\ZZ) \oplus \ZZ_2$ by methods from homotopy theory. We will provide
  a geometrical description of the $\ZZ_2$ part in $\pi^n(M)$ analogous to Pontryagin's
  computation of the stable homotopy group $\pi_{n+1}(S^n)$. This $\ZZ_2$ number can
  be computed by counting embedded circles in $M$ with a certain framing of their normal
  bundle. This is a similar result to the mod $2$ degree theorem for maps $M \to S^{n+1}$.

  Finally we will observe that the zero locus of a section in an oriented rank $n$ vector bundle
  $E \to M$ defines an element in $\pi^n(M)$ and it turns out that the $\ZZ_2$ part is an invariant of
  the isomorphism class of $E$. At the end we show, that if the Euler class of $E$ vanishes this
  $\ZZ_2$ invariant is the final obstruction to the existence of a nowhere vanishing section.
\end{abstract}
\section{Introduction}\label{S:Introduction}
\emph{Pontryagin} computed in \cite{MR0115178} the (stable) homotopy group $\pi_{n+1}(S^n)$ ($n\geq 3$) by
using differential topology. Let us describe briefly his construction, since this paper will
generalize his idea.

Pontryagin showed that $\pi_{n+1}(S^n)$ is isomorphic to the bordism group of closed
$1$-dimensional submanifolds of $\RR^{n+1}$ furnished with a framing on its normal bundle (a \emph{framing}
is a homotopy class of trivializations, see section \ref{S:Preliminaries}). We denote
this bordism group by $\Omega_1^{\textrm{fr}}(\RR^{n+1})$. Let $(C,\varphi)$ be a representative
of an element of $\Omega_1^{\textrm{fr}}(\RR^{n+1})$, i.e. $C$ is a union of embedded circles in
$\RR^{n+1}$ and there are maps $\varphi_1,\ldots,\varphi_n \colon C \to \RR^{n+1}$ such that
$(\varphi_1(x),\ldots,\varphi_n(x))$ is a basis of $\nu(C)_x$ for every $x \in C$. Let
$\varphi_{n+1}$ be a trivialization of the tangent bundle of $C$. Then
$(\varphi_1(x),\ldots,\varphi_{n+1}(x))$ is a basis of $\RR^{n+1}$ for every $x \in C$. Without
loss of generality we may assume that $\varphi_1,\ldots,\varphi_{n+1}$ is pointwise an orthonormal
basis. If $(e_1,\ldots,e_{n+1})$ denotes the standard basis of $\RR^{n+1}$ then consider the map
$A=(a_{ij}) \colon C \to \mathrm{SO}(n+1)$ such that
\[
  \varphi_i(x) = \sum_{j=1}^{n+1} a_{ij}(x)e_j
\]
for $x \in C$. Let $\pi_1(\mathrm{SO}(n+1))$ be identified with $\ZZ_2$, then Pontryagin defines
\cite[Theorem 20]{MR0115178}
\[
  \delta(C,\varphi) := [A] + (n(C) \mod 2)
\]
where $[A]$ denotes the homotopy class of $A$ in $\pi_1(\mathrm{SO}(n+1))$ and $n(C)$ is the number
of connected components of $S$. He showed that $\delta$ is well-defined on
$\Omega_1^{\textrm{fr}}(\RR^{n+1})$ and is an isomorphism of groups.

From a different point of view, one may consider his computation not as a computation of a homotopy
group of $S^n$ but rather of a \emph{cohomotopy group} of $S^{n+1}$. If $X$ is a CW space then the
cohomotopy set of $X$ is defined as the set of (unrestricted) homotopy classes
$\pi^n(X):=[X,S^n]$, cf. \cite{MR0029170, MR1549930}. The set $\pi^n(X)$ for $X$ a finite CW complex of
dimension $n+1$ carries naturally a group structure, which is described in the beginning of section
\ref{S:computation of pinM}.
\emph{Steenrod} showed \cite[Theorem 28.1, p. 318]{MR0022071} that $\pi^n(X)$ fits into a short exact
sequence
\[
  0 \longrightarrow H^{n+1}(X;\ZZ_2)/ \Sq^2 \mu(H^{n-1}(X;\ZZ)) \longrightarrow
  \pi^n(X) \longrightarrow H^n(X;\ZZ) \longrightarrow 0,
\]
where $\mu \colon H^{\ast}(X;\ZZ) \to H^\ast(X;\ZZ_2)$ is the mod $2$
reduction homomorphism. Here the surjective map is the Hurewicz homomorphism which assigns to every $f \in \pi^n(X)$ the
cohomology class $f^\ast(\sigma) \in H^n(X;\ZZ)$ where $\sigma \in H^n(S^n,\ZZ)$ is a fixed generator.

Moreover using methods of \emph{Larmore} and \emph{Thomas} \cite{MR0328935} \emph{Taylor} showed in
\cite[Theorem 6.2, Example 6.3]{MR3084240} that the short exact sequence splits, provided the images of $\Sq^2 \colon H^{n-1}(X;\ZZ_2)
\to H^{n+1}(X;\ZZ_2)$ and $\Sq^2 \circ \mu \colon H^{n-1}(X;\ZZ) \to H^{n+1}(X;\ZZ_2)$ coincide.

If $X=M$ is a manifold then the \emph{second Wu class} \cite{MR35992} is equal to the
second Stiefel-Whitney class $w_2(M)$, hence $\Sq^2(x) = w_2(M)\smile x$ for $x \in
H^{n-1}(M;\ZZ_2)$. Therefore if $M$ is spin then $\pi^n(M)$ fits into the exact sequence
\[\label{eq:SES}
  0 \longrightarrow\ZZ_2 \longrightarrow
  \pi^n(M) \longrightarrow H^n(M;\ZZ) \longrightarrow 0. \tag{ST}
\]
and \eqref{eq:SES} splits by \cite[Example 6.3]{MR3084240} thus
\[
  \pi^n(M) \cong H^n(M;\ZZ) \oplus \ZZ_2
\]
as abelian groups. However the splitting map is constructed in a purely homotopy theoretic setting
and an aim of this article is to provide a geometric description in case $M$ is a spin manifold.

This splitting map $\kappa \colon \pi^n(M) \to \ZZ_2$ (see Definition \ref{D:invariant}) for
\eqref{eq:SES} will be constructed similarly to Pontryagin's invariant $\delta$ from above. An
important ingredient in Pontryagin's construction was the canonical \emph{background framing} by
the standard basis of $\RR^{n+1}$, which allowed him to define the map $A\colon S \to
\mathrm{SO}(n+1)$. In general if we replace $S^{n+1}$ or $\RR^{n+1}$ by $M$, this background framing
is not available any more. But this can be circumvented by using the spin structure of $M$, since
over a circle every vector bundle with a spin structure defines a certain framing, cf. Lemma \ref{L:Framing
on C}. Section \ref{S:computation of pinM} is devoted to determine geometrically the kernel
of the Hurewicz map $\pi^n(M) \to H^n(M;\ZZ)$. Finally we show that the splitting map possesses
a naturality property, cf. Proposition \ref{P:naturality} and that for a map $f\colon M\to S^n$ the number
$\kappa(f)$ can be described by a counting formula, cf. Corollary \ref{C:counting formula}. This is an
analogous result to the mod $2$ Hopf theorem, see \cite[§4]{MR0226651}. It should be mentioned
that in \cite{MR3084237} the authors discuss the case $n=3$ and in \cite{1812.06547v2} a similar
construction of a $\ZZ_2$ invariant was used to classify quaternionic line bundles over closed spin $5$-manifolds.

In Section \ref{S:Application to vector bundles} we will apply the results of Sections \ref{S:Index
of framed circles} and \ref{S:computation of pinM} to the theory of vector bundles. Suppose $E \to
M$ is a oriented vector bundle of rank $n$ over a closed spin $(n+1)$-manifold $M$. Then any
section of $E$ which is transverse to the zero section defines by means of its zero locus an
element of $\Omega_1^{\textrm{fr}}(M)$ and this element is independent of the transverse section.
Thus using $\kappa$ one defines an invariant $\kappa(E) \in \Omega_1^\textrm{fr}$ of the
isomorphism class of the bundle
$E\to M$. In Theorem \ref{T:now where vanishing section} it is shown, that $\kappa(E)$ can be
regarded as the secondary obstruction to the existence of a nowhere vanishing section. As an
application we provide in Example \ref{E:Spheres} a simple proof of the well-known fact, that the
maximal number of linear independent vector fields on $S^{4k+1}$ is equal to $1$. Finally we show
that $\pi^n(M)$ can be mapped injectively into the set of isomorphism classes of oriented rank $n$
vector bundles over spin $(n+1)$-manifolds for $n=4$ and $n=8$, cf. Proposition \ref{P: vector
bundles in special cases}.

\section*{Acknowledgements}\label{S:Acknowledgements}
The author would like to thank S. Carmeli and M. Miller for helpful discussions in \cite{307928}.
\section{Preliminaries}\label{S:Preliminaries}
If not otherwise stated we denote by $M$ an $(n+1)$-dimensional oriented, closed and connected
manifold, where $n\geq 3$. Let $N$ be a arbitrary manifold and $E \to N$ a trivial vector bundle
over $N$ of rank $r$. A \emph{trivialization} of $E\to N$ are $r$ sections $s_1,\ldots,s_r \colon
N \to E$, such that $(s_1(q),\ldots,s_r(q))$ is a basis of the fiber $E_q$ for all $q \in N$. A
\emph{framing} $\varphi$ of $E\to N$ is a homotopy class of trivializations.

We recall now the notion of bordism classes of normally framed submanifold in $M$ of
dimension $k$ (cf. \cite[§7]{MR0226651}). Let $C$ be
a $k$-dimensional closed submanifold of $M$. We say that $C$ is \emph{normally framed} if the
normal bundle of $C$ is trivial and possesses a framing $\varphi$. Two such normally framed
submanifolds $(C_0,\varphi_0)$ and $(C_1,\varphi_1)$ are \emph{framed bordant} if there is
a $(k+1)$-dimensional submanifold $\Sigma \subset M \times [0,1]$ such that
\begin{enumerate}[label=(\alph*)]
  \item $\partial \Sigma \cap (M \times i) = C_i$ for $i=0,1$,
  \item $\partial \Sigma = C_0 \cup C_1$,
  \item $\Sigma$ is normally framed in $M \times [0,1]$ such that the framing restricted to the
    $\partial \Sigma \cap (M \times i)$ coincides with $\varphi_i$.
\end{enumerate}
To be framed bordant is an equivalence relation and the set of equivalence classes is called the
\emph{bordism classes of normally framed $k$-dimensional submanifolds} denoted by $\Omega_k^{\textrm{fr}}(M)$.
If $(C,\varphi)$ is a normally framed submanifold then we denote by $[C,\varphi]$ its bordism
class in $\Omega_k^{\textrm{fr}}(M)$.

The \emph{Pontryagin-Thom map} provides a bijection
between $\pi^{n+1-k}(M)$ and $\Omega_k^{\textrm{fr}}(M)$ as follows (cf. \cite[§7]{MR0226651}): Let
$f \colon M \to S^{n+1-k}$ represents an element of $\pi^{n+1-k}(M)$. Choose a regular value $x_0
\in S^{n+1-k}$ and set $C_{x_0}:= f^{-1}(x_0)$. Moreover choosing a basis of the tangent space
$T_{x_0}S^{n+1-k}$ endows the normal bundle with a framing $\varphi_{x_0}$ by means of the
derivative of $f$. The bordism class $[C_{x_0},\varphi_{x_0}] \in \Omega_{k}^{\textrm{fr}}(M)$ is
well defined and the map
\[
  \pi^{n+1-k}(M)\longrightarrow \Omega_k^{\textrm{fr}}(M), \quad [f]\mapsto
  [C_{x_0},\varphi_{x_0}].
\]
is a bijection, see \cite[Theorem B and A]{MR0226651}.

A \emph{stable framing} of a real vector bundle $E \to C$ of rank $r$ is an equivalence class of trivializations of
\[
  E \oplus \varepsilon^l
\]
for some $l \in \NN$ where two trivializations
\[
  \tau_1 \colon E \oplus\varepsilon^{l_1} \to \varepsilon^{r+l_1}
  \quad \textrm{and}\quad
  \tau_2 \colon E \oplus\varepsilon^{l_2} \to \varepsilon^{r+l_2},
\]
are considered to be equivalent if there exists some $L>l_1,l_2$ such that the isomorphisms
\[
  \tau_1 \oplus \textrm{id}\colon E \oplus\varepsilon^{l_1} \oplus \varepsilon^{L-l_1} \to \varepsilon^{L+r}
\]
and
\[
  \tau_2 \oplus \textrm{id}\colon E \oplus\varepsilon^{l_2} \oplus \varepsilon^{L-l_2} \to \varepsilon^{L+r}
\]
are homotopic, cf. \cite[Section 8.3]{MR1841974}. If $E$ is the tangent bundle of $C$, then
a stable framing of $TC$ is called a \emph{stable tangential framing}. If $E$ is the normal bundle
of an embedding of $C$ into a sphere of big dimension, then we call a stable framing a \emph{stable normal framing}.

  We define $\Omega_k^{\textrm{fr}}$ to be the bordism classes of stably (tangential) framed
  manifolds. More precisely two stably framed manifolds $(C_0,\varphi_0)$ and $(C_1,\varphi_1)$
  where $\varphi_i \colon TC_i \oplus \varepsilon^l \to \varepsilon^{k+l}$ is an isomorphism are
  equivalent if there is a bordism $\Sigma$ between $C_0$ and $C_1$ such that the tangent bundle of
  $\Sigma$ possesses a stable framing and the restriction on $C_0$ and $C_1$ coincides with the
  framing $\varphi_0$ and $\varphi_1$ respectively. Note that $\Omega_k^{\textrm{fr}}$ is
  isomorphic to $\pi_k^S$, the $k$-stable homotopy group of spheres (cf. \cite[Theorem
  8.17]{MR1841974}) and by the Pontryagin-Thom construction we have $\Omega_k^{\textrm{fr}}
  = \varinjlim_l \Omega_k^{\textrm{fr}}(S^l)$ where we use the equatorial embeddings $S^{l_1}
  \hookrightarrow S^{l_2}$ if $l_1 < l_2$ to construct well-defined maps
  $\Omega_k^{\textrm{fr}}(S^{l_1}) \to \Omega_k^{\textrm{fr}}(S^{l_2})$.

For this article the case $k=1$ will be of importance. In this case we have $\Omega_1^{\textrm{fr}}\cong
\pi_1^S \cong \ZZ_2$. Consider a connected and closed $1$-dimensional
manifold $S_0$ and stable tangential framing
\[
  \varphi_0 \colon TS_0 \oplus \varepsilon^n \stackrel{\sim}{\longrightarrow} \varepsilon^{n+1}.
\]
From the discussion above, $(S_0,\varphi_0)$ defines a class in $\Omega_1^{\textrm{fr}}$ and can be
realized as follows: Consider $S_0 =\{(x_1,\ldots,x_{n+1}) \in \RR^{n+1} : x_1^2+x_2^2=1,\, x_i=0,\,
i=3,\ldots,n+1\}$. Denote $e_1,\ldots,e_{n+1}$ the canonical basis of $\RR^n$ and $E_i(x)=e_i$ for
$x\in \RR^{n+1}$ the constant vector fields on $\RR^{n+1}$. Moreover let $V(x)=x$ for $x \in
\RR^n$.
 The normal bundle $\nu(S_0)$ of
$S_0$ is trivialized by $V,E_3,\ldots,E_{n+1}$ restricted to $S_0$. Using this normal framing we obtain
a stable framing
\[
  TS_0 \oplus\varepsilon^{n} \cong TS_0 \oplus \nu(S_0) \cong (T\RR^{n+1})|_{S_0} \cong \varepsilon^{n+1}
\]
where the latter framing is induced by $E_1,\ldots,E_{n+1}$. Hence this defines an element in
$\Omega_1^{\textrm{fr}}(S^{n+1})$ which represents the framed null bordism, since the framing of $\nu(S_0)$
can be extended to a properly embedded stably framed disc in $S^{n+1} \times [0,1]$. Clearly the non-trivial element
of $\Omega_1^{\textrm{fr}}(S^{n+1})$ can be represented by twisting the normal framing $E_3,\ldots,E_{n+1}$ with
a map $S_0 \to \textrm{SO}(n)$ such that its homotopy class in $\pi_1\left( \textrm{SO}(n)
\right)\cong \ZZ_2$ is not zero. Every stable tangential framing of a closed and connected
$1$-dimensional manifold can be obtain in this way.

If $E \to N$ is an oriented vector bundle over a manifold $N$, then we say that $E$ is \emph{spinnable}
if the second Stiefel-Whitney class $w_2(E)$ is zero. This means that $E$ can carry a spin
structure, that is a lift of the classifying map $N \to B \textrm{SO}(n)$ to a map $N \to B \textrm{Spin}(n)$
in the fibration $K(\ZZ_2,1) \to B \textrm{Spin}(n) \to B \textrm{SO}(n)$. Consequently
$E$ is a \emph{spin bundle} is it spinnable and a spin structure is fixed.
If a spin structure is fixed on $E \to N$ then any other spin structure is in $1 : 1$ correspondence
with elements in $H^1(N;\ZZ_2)$.

We write $F(N)$ for the frame bundle of a manifold $N$. If $V\subset N$ is a submanifold such that
its normal bundle is framed then we obtain an embedding $F(V) \subset F(N)$. Thus a spin structure
on $N$ induces a spin structure on $V$, cf. \cite{MR0157388}. In particular if $V$ is the boundary
of a spin manifold $N$, then $V$ inherits a spin structure
from $N$. Finally if $E \to N$ is a vector bundle with a spin structure and $V\subset N$
a submanifold, then clearly $E|_V \to V$ also inherits a spin structure from $E \to N$.

Let $E \to S^1$ be a spinnable vector bundle of rank $r \geq 3$ over the unit circle $S^1$. Then
$E$ has exactly two non-isomorphic spin structures. Clearly $E\to S^1$ can be extended to $E \to D^2$, where $D^2$ denotes the closed
unit disc in $\RR^2$. Since $D^2$ is contractible $E\to D^2$ admits a unique spin structure.
Restricting this structure to the boundary of $D^2$ gives a spin structure on $E\to S^1$, which
will be called the \emph{standard spin structure}. The other should be called the
\emph{non-standard spin structure}. In other words, the standard spin structure on $E \to S^1$ can
be extended to $D^2$, the non-standard not.
\section{The index of framed circles}\label{S:Index of framed circles}
We define in this section the key invariant of this article. For its construction
the following basic lemma is the crucial observation.

\begin{lem}\label{L:Framing on C}
Let $E \to S^1$ be a spinnable vector bundle of rank $\geq 3$. Then $E$ is isomorphic to the
  trivial bundle and a choice of a spin structure on $E$ determines a framing on $E$.
\end{lem}
\begin{proof}
  $E$ is isomorphic to the trivial bundles since it is an orientable vector bundle over a circle.
  Fix a spin structure on $E$, i.e. let $F'(E)$ be a $\mathbf{Spin}(n)$-principal bundle over $S^1$
  which is a two-sheeted cover over the frame bundle $F(E)$ of $E$. Let $\pi\colon F'(E)\to F(E)$
  be the projection which is equivariant with respect to the two-sheeted covering
  $\mathbf{Spin}(n)\to \SOg(n)$. Clearly $F'(E)$ is the trivial $\mathbf{Spin}(n)$-principal bundle
  over $S^1$ and denote by $\sigma \colon S^1 \to F'(E)$ a global section. Then $\pi\circ \sigma$
  is a global section of $F(E)$ hence a trivialization of $E \to S^1$. Any other such global section
  $\widetilde\sigma \colon S^1 \to E$ differs from $\sigma$ by a map $\varphi\colon S^1 \to
  \mathbf{Spin}(n)$. Since $\pi_1(\mathbf{Spin}(n))=1$ the map $\varphi$ has to be null-homotopic
  which means that the two trivializations $\pi\circ\sigma$ and $\pi\circ\widetilde\sigma$ have to
  be homotopic, thus they define the same framing on $E$.
\end{proof}
In the same way one proves
\begin{cor}\label{C:1-dim CW complex}
  Let $\Sigma$ be a $1$-dimensional CW-complex (not necessarily connected) and $E\to \Sigma$ a vector bundle of rank $\geq 3$
endowed with a spin structure. Then $E$ is isomorphic to the trivial bundle and the spin structure
induces a framing on $E$.
\end{cor}

\begin{dfn}\label{D:standard framing}
Let $E \to S^1$ be a spinnable vector bundle. The framing induced by the standard spin structure
  on $E$ is called the \emph{standard framing} and from then non-standard spin structure
  the \emph{non-standard framing}.
\end{dfn}

\begin{exmpl}\label{ex: sphare triv}
The spheres $S^{n+1}$ admit a unique spin structure which can be constructed as mentioned
in the preliminaries, i.e. $S^{n+1}$ is the boundary
of the closed unit ball $D^{n+2}$ in $\RR^{n+2}$ which admits a unique spin structure.

Let $S_0 \in S^{n+1}$ be the intersection of a $2$-dimensional linear subspace $W\subset \RR^{n+2}$
with $S^{n+1}$ and denote by $D_0^2 = W \cap D^{n+2}$. Thus after Lemma \ref{L:Framing on C}
$TS^{n+1}|_{S_0}$ inherits a framing from the spin structure. Denote by $\varphi_1,\ldots,\varphi_{n+1}$
a trivialization of this framing, then the framing
\[
  \overline\varphi\colon S_0 \to \textrm{SO}(n+2),\quad x \mapsto (x,\varphi_1(x),\ldots,\varphi_{n+1}(x))
\]
  must be null homotopic in $\textrm{SO}(n+2)$ by the definition of the spin structures of
  $S^{n+1}$ and $TS^{n+1}|_{S_0}$ (such that it lifts to $\textrm{Spin}(n+2)$). Thus
  $\overline\varphi$ must be homotopic the constant framing $x \mapsto (e_1,\ldots,e_{n+2})$, where
  $e_1,\ldots,e_{n+2}$ denotes the canonical basis of $\RR^{n+2}$. In particular this means,
  that $TS^{n+1}|_{S_0}$ inherits the standard framing from the spin structure of
  $S^{n+1}$.
\end{exmpl}

$\Omega_1^{\textrm{fr}}(M)$ possesses a group structure which can be expressed as follows: Having
two $1$-dimensional closed submanifolds $C$ and $C'$ of $M$ which are normally framed, then they
are framed bordant in $M$ to framed submanifolds $\tilde C$ and $\tilde C'$ whose
intersection is empty. Taking the equivalence class of the disjoint union $\tilde C \cup \tilde C'$ with the
respective framings yields an abelian group structure on $\Omega_1^{\textrm{fr}}(M)$, cf.
\cite[Problem 17 and p. 50]{MR0226651}.

Next, we construct a homomorphism $\kappa \colon \Omega_1^{\textrm{fr}}(M) \to \Omega_1^{\textrm{fr}}$,
where $\Omega_1^{\textrm{fr}}$ is  the bordism group of stably framed $1$-dimensional closed
manifolds. Therefore let $(C,\varphi_C)$ be a closed submanifold of dimension $1$,
such that its normal bundle $\nu(C)$ is framed by $\varphi_C$ (thus representing an element in
$\Omega_1^{\textrm{fr}}(M)$). From Lemma \ref{L:Framing on C} the bundle $TM|_C$ inherits
a framing $\varphi_\sigma$ from the spin structure of $M$. Using also the framing of $\varphi_C$ we
obtain a stable tangential framing

\[
  \varepsilon^{n+1} \cong TM|_C \cong TC \oplus \nu(C) \cong TC \oplus \varepsilon^n
\]
which we denote by $\varphi_{\textrm{st}}$.
\begin{prop}\label{P:kappa well defined}
  The bordism class $[C,\varphi_{\textrm{st}}] \in \Omega_1^{\textrm{fr}}$ depends
  only on the bordism class of $[C,\varphi_C] \in \Omega_1^{\textrm{fr}}(M)$.
\end{prop}

\begin{proof}
  Let $(C',\varphi_{C'})$ be another normally framed closed $1$-dimensional submanifold framed
  bordant to $(C,\varphi_C)$. Thus there is a bordism $\Sigma \subset M \times I$ between
  $C$ and $C'$ such that the normal bundle of $\Sigma$ in $M\times I$ possess a framing
  $\varphi_\Sigma$. By definition restricting $\varphi_\Sigma$ to $C$ and $C'$ yields
  $\varphi_C$ and $\varphi_{C'}$ respectively. Since $\Sigma$ is homotopy equivalent to
  a $1$-dimensional CW-complex and since $M\times I$ inherits a unique spin structure from $M$
  we obtain a framing $\varphi_{\Sigma,\sigma}$ on $T(M\times I)|_{\Sigma}$. Of course the framings
  $\varphi_{\Sigma,\sigma}$ restricted to $C$ and $C'$ are just the framings $\varphi_\sigma$ and
  $\varphi_\sigma'$ respectively (i.e. induced by the spin structure of $TM|_C$ and $TM|_{C'}$).
  Since
  \[
    T(M\times I)|_\Sigma \cong T\Sigma \oplus \nu(\Sigma)
  \]
  the framings $\varphi_{\Sigma,\sigma}$ and $\varphi_\Sigma$ determine a stable framing $\varphi_{\Sigma,\textrm{st}}$ of $T\Sigma$.
  Then $(\Sigma,\varphi_{\Sigma,\textrm{st}})$ is stably framed bordism between $(C,\varphi_{\textrm{st}})$
  and $(C',\varphi_{\textrm{st}}')$.
\end{proof}

\begin{rem}\label{R:cricles generated bordism group}
  As described above, the group structure of $\Omega_{1}^{\textrm{fr}}(M)$ is given by disjoint
  union of submanifolds and their respective normal framings. Let $(C,\varphi)$ be a framed
  $1$-dimensional closed submanifold of $M$ and denote by $C=S_1 \cup \ldots \cup S_k$ the
  connected components of $C$. We may assume that the union is always disjoint. Thus $S_i$ is an
  embedded circle and $\varphi_i := \varphi|_{S_{i}}$ a normal framing of $S_i$. Consequently we
  have
  \[
    [C,\varphi] = \sum_{i=1}^{k}\; [S_i,\varphi_i].
  \]
\end{rem}

\begin{dfn}\label{D:framed circle}
Let $S \subset M$ be an embedded circle and $\varphi$ a framing of $\nu(S)$. We call the bordism
  class $[S,\varphi] \in \Omega_1^{\textrm{fr}}(M)$ a \emph{framed circle of $M$}.
  The corresponding stable class $[S,\varphi_{\textrm{st}}]\in \Omega_1^{\textrm{fr}}$ will be
  called the \emph{the index of $[S,\varphi]$ (with respect to the spin structure of $M$)}
  and will be denoted by $\operatorname{ind}(S,\varphi)$.
\end{dfn}

\begin{dfn}\label{D:invariant}
  Let $M$ be an $(n+1)$-dimensional closed spin manifold. Then we define
  a map
  \[
    \kappa \colon \Omega_1^{\textrm{fr}}(M) \to \Omega_1^{\textrm{fr}},\quad
    [C,\varphi] \mapsto \kappa\left( [C,\varphi]  \right) := \sum_{S \subset C,\atop S\textrm{ connected}}^{}
    \operatorname{ind}(S,\varphi|_S) = [C,\varphi_{\textrm{st}}].
  \]
  We call $\kappa$ the \emph{degree map} of $M$ with respect to the chosen spin structure.
\end{dfn}

\begin{rem}\label{R:kappa homo}
It is clear from the construction that $\kappa$ is a homomorphism.
\end{rem}

\begin{exmpls}\label{Ex:First examples}
\begin{enumerate}[label=(\alph*)]
  \item The spheres $S^{n+1}$ admit a unique spin structure which is induced by
    the closed $(n+2)$-dimensional disc $D^{n+2} \subset \RR^{n+2}$, cf. \cite{MR0157388}.

    Let $S_0$ be the intersection of $S^{n+1}$ with a $2$-dimensional linear subspace $W$ of
    $\RR^{n+2}$. We argued in Example \ref{ex: sphare triv} that $TS^{n+1}|_{S_0}$ inherits
    the standard framing.

    Choose the standard framing $\varphi_0$ on $\nu(S_0)$, then
    \[
      \kappa([S_0,\varphi_0])=0.
    \]
    Consequently the non-standard framing $\varphi_1$ of $\nu(S_0)$ yields
    \[
      \kappa([S_0,\varphi_1])\neq 0.
    \]

  \item\label{ex: depending on spin structure} Let $N$ be a closed, simply connected, spin manifold of dimension $n$. Then $M:=S^1 \times N$
    admits two different spin structures since $H^1(S^1 \times N;\ZZ_2) \cong H^1(S^1;\ZZ_2) \cong
    \ZZ_2$. $M$ is the boundary of $D^2 \times N$ which has up to isomorphism a unique spin
    structure. The two different spin structures on $M$ can be described as follows: One can
    be extended from $M$ to $D^2 \times N$ and the other not. We call the latter one the
    \emph{standard spin structure} and the former one the \emph{non-standard spin structure} of
    $S^1 \times N$.

    For $q_0 \in N$ consider the circle $S_0:= S^1 \times q_0 \subset S^1 \times N$. Clearly
    we have a canonical isomorphism
    \[
      \nu(S_0) \cong S_0 \times T_{q_0}N.
    \]
    Thus choosing a basis in $T_{q_0}N$ gives a framing $\varphi_0$ on $\nu(S_0)$ which extends
    to a framing of $(D^2 \times q_0) \times T_{q_0}N$. Thus we have
    \[
      \kappa_0([S_0,\varphi_0]) =0
    \]
    for the standard spin structure and
    \[
      \kappa_1([S_0,\varphi_0]) \neq 0
    \]
    for the non-standard spin structure.

    For $q_1 \in N$ with $q_0 \neq q_1$ we consider $C = S^1 \times q_0 \cup S^1 \times q_1$
    with fixed normal framing on $S^1 \times q_i$ which gives a framing $\varphi$ on $C$.
    Then $\kappa([C,\varphi])$ is independent of the chosen spin structure of $M$.

    This shows that in general $\kappa$ will depend
    on the spin structure. The next proposition will show how it depends from it.
\end{enumerate}
\end{exmpls}
%

  Suppose $C\subset M$ is a closed $1$-dimensional submanifold. Then $C$ defines a $\ZZ_2$
  fundamental homology class $[C] \in H_1(M;\ZZ_2)$. We denote by $w(C) \in H^n(M;\ZZ_2)$ the
  cohomology class which is the Poincar\'{e} dual of $[C]$.
\begin{prop}\label{P: dependence on spin structure}
  Fix a spin structure $\sigma$ on $M$ and denote by $\kappa$ the degree map of $M$ with respect to the
  chosen spin structure. Choose another spin structure of $M$, which is represented by $\alpha
  \in H^1(M;\ZZ_2)$ and denote by $\kappa^\alpha$ the corresponding degree map. Then we have
  \[
    \kappa([C,\varphi]) = \kappa_\alpha([C,\varphi]) + \delta(\alpha \smile w(C)),
  \]
  where $\delta\colon H^{n+1}(M;\ZZ_2) \to \Omega_1^{\textrm{fr}}$ is the unique isomorphism. Thus
  if $w(C) \smile \alpha=0$ then $\kappa\left( [C,\varphi]  \right)=\kappa^{\alpha}\left( [C,\varphi]  \right)$.
\end{prop}

\begin{proof}
  Assume first that $(S,\varphi)$ is a framed circle and $i \colon S \to M$ is the inclusion. The spin structure
  $\sigma$ induces a spin structure on $TM|_S = i^\ast(TM)$ and the spin structure induced
  by $\alpha$ is represented by $i^\ast(\alpha) \in H^1(S;\ZZ_2)$. Of course $TM|_S$ can have at
  most two different spin structures. From the definition of the index we have
  \[
    \operatorname{ind}(S,\varphi) = \operatorname{ind}_\alpha(S,\varphi) + \delta(i^\ast(\alpha))
  \]
  where $\operatorname{ind}$ is defined by $\sigma$, $\operatorname{ind}_\alpha$ by $\alpha$
  and $\delta \colon H^1(S;\ZZ_2) \to \Omega_1^{\textrm{fr}}$ the unique isomorphism.

  Let $[S] \in H_1(S;\ZZ_2)$ be the $\ZZ_2$ fundamental
  class of $S$, then $i^\ast(\alpha)\frown [S] \in H_0(S;\ZZ_2)$, which is mapped under $i_\ast$
  to $\alpha \frown i_\ast([S]) \in H_0(M;\ZZ_2)$. Let $[M] \in H_n(M;\ZZ_2)$ denote the $\ZZ_2$
  fundamental class of $M$. Then we compute
  \[
    \alpha \frown i_\ast([S]) = \alpha \frown \left( w(S) \frown [M] \right)
    =\left( \alpha\smile w(S)  \right)\frown [M],
  \]
  where we used that $i_\ast([S])$ is Poincar\'e dual to $w(S)$. Since $ \cdot  \frown [S]$ and
  $ \cdot \frown [M]$ are isomorphisms by Poincar\'{e} duality and $i_\ast \colon
  H_0(S;\ZZ_2) \to H_0(M;\ZZ_2)$ is also an isomorphisms because $S$ and $M$ are connected we infer
  \[
    \operatorname{ind}(S,\varphi) = \operatorname{ind}_\alpha(S,\varphi) + \delta(\alpha\smile w(S))
  \]
  where now $\delta \colon H^{n+1}(M;\ZZ_2) \to \Omega_1^{\textrm{fr}}$ is again the unique
  isomorphism.

  Consider now $(C,\varphi)$  with the disjoint union $C= S_1 \cup \ldots S_k$ and $\varphi_j := \varphi|_{S_j}$,
  such that $S_j$ is connected. With the previous computations we have
  \[
    \kappa([C,\varphi]) =\sum_{j=1}^{k} \left( \operatorname{ind}_\alpha(S_j,\varphi_j) +\delta(\alpha
    \smile w(S_j)  \right) = \kappa_\alpha([C,\varphi]) + \delta(\alpha \smile w(C)).
  \]
  and the proposition follows.
\end{proof}
%
We continue with the description of the \glqq dual\grqq\, short exact sequence to \eqref{eq:SES}.
There is a natural group homomorphism $\Omega_1^{\textrm{fr}}(M) \to \Omega_1^{\textrm{SO}}(M)$,
which assigns to every framed $1$-submanifold $[C,\varphi]$ the oriented bordism class induced by
the orientation framing $\varphi$. This is well-defined since every normally framed bordism in $M$ is also an
oriented bordism ($M$ is oriented). By the seminal work of \emph{Thom} \cite{MR0061823} we have an
isomorphism
\[
  \Omega_1^{\textrm{SO}}(M) \to H_1(M;\ZZ)
\]
which assigns every oriented submanifold its fundamental class in $H_1(M;\ZZ)$. Thus we obtain
a group homomorphism
\begin{equation}\label{eq:Phi}
  \Phi \colon \Omega_1^{\textrm{fr}}(M) \to H_1(M;\ZZ)
\end{equation}
which is clearly surjective. The kernel of $\Phi$ is at most isomorphic to $\ZZ_2$ and elements
of the kernel are represented by framed circles $(S,\varphi)$ such that $S$ is oriented null-bordant,
i.e. there is an embedded oriented disc $D\subset M \times I$ with the properties $\partial D =S$ and
the orientations of $\partial D$ and $S$ agree. We may equip the normal bundle of $S$ with two
framings. If both framings can be extended over $D$ then the kernel is trivial and otherwise
$\ZZ_2$.

\begin{lem}\label{L:kernerl isomorphic to omega1fr}
  The restricted degree map $\kappa|_{\ker\Phi} \colon \ker\Phi \to \Omega_1^{\textrm{fr}}$ is an isomorphism.
\end{lem}
\begin{proof}
 Since $\kappa$ is a homomorphism it will map the neutral element of $\ker\Phi$ to that of
  $\Omega_1^{\textrm{fr}}$. Thus it suffices to show the
  following: Let $(S,\varphi)$ be a framed circle such that $S$ is oriented null-bordant in $M$ but
  $\varphi$ cannot be extended over the nullbordism. We have to show $\kappa([S,\varphi])\neq 0$,
  where $0$ denotes the neutral element of $\Omega_1^{\textrm{fr}}$. We may assume that $S$ lies
  in a chart of $M$
  \footnote{Take a small embedded closed disc and choose a framing on the circle bounding the disc
  which cannot be extended over a proper embedded disc in $M \times I$.}.
  Thus we may embed $S$ into $\RR^{n+1}$ endowed
  with a normal framing, which cannot be extended over a nullbordism in $\RR^{n+1}$. Hence
  the index of $(S,\varphi)$ defines a non-trivial element in $\Omega_1^{\textrm{fr}}$ (note
  that since $w(S)=0$ the element $\kappa[(S,\varphi)]$ does not depend on the spin structure of
  $M$, cf. Lemma \ref{P: dependence on spin structure}).
\end{proof}

Thus we may identify $\ker\Phi$ with $\Omega_1^{\textrm{fr}}$ via $(\kappa|_{\ker\Phi})^{-1}$ and
we obtain a short exact sequence
\[
  0 \longrightarrow \Omega_1^{\textrm{fr}} \longrightarrow \Omega_1^{\textrm{fr}}(M)
  \longrightarrow H_1(M;\ZZ) \longrightarrow 0
\]
and from Lemma \ref{L:kernerl isomorphic to omega1fr} $\kappa$ is a splitting map. Therefore

\begin{thm}\label{T:BordismGroupSplits}
Let $M$ be an $(n+1)$-dimensional closed spin manifold. Choose a spin structure on $M$. Then
\[
  \Omega_1^{\textrm{fr}}(M) \longrightarrow H_1(M;\ZZ) \oplus \Omega_1^{\textrm{fr}},
  \quad [C,\varphi] \mapsto ([C],\kappa([C,\varphi]))
\]
is an isomorphism of abelian groups.
\end{thm}

We finish this section by giving an alternative way to compute the index of a framed circle in the
spirit of Pontryagin \cite{MR0115178}.
Suppose $[S,\varphi]$ is a framed circle, thus there are trivializations of $\nu(S)$ and $TM|_{S}$
such that we obtain the stable framing
\[
  \varepsilon^{n+1} \cong TS \oplus \varepsilon^n
\]
(where is can assume that the isomorphism is orientation preserving). Denote by $v_1,\ldots,v_{n+1}$
and by $w_2,\ldots,w_{n+1}$ the trivializations of $TM|_S$ and $\nu(S)$ respectively. Let $w_1$ be
a trivialization of $TS$. Let $\Phi \colon TS \oplus \varepsilon^n \to \varepsilon^{n+1}$ be the
isomorphism of the stable framing, then there is a matrix $A=(A_{ij}) \colon S \to
\mathrm{GL}^{+}(n+1)$ (where $\textrm{GL}^+(n+1)$ is the set of all invertible real matrices of
size $(n+1) \times (n+1)$ with positive determinant) such that such that
\[
  \Phi(w_i) = \sum_{j=1}^{n+1} A_{ij} \cdot v_j.
\]
Since $\mathrm{SO}(n+1)$ is a strong deformation retract of $\mathrm{GL}^+(n+1)$ we have $\pi_1
\left(\mathrm{GL}^+(n+1)\right) \cong \ZZ_2$. The map $A \colon  S \to \mathrm{GL}^+(n+1)$ defines
an element $[A] \in \pi_1\left( \mathrm{GL}^+(n+1)  \right)$. Changing the homotopy classes of
trivializations of $TM|_S$ and $\nu(S)$ does not change $[A]$. Furthermore $[A]$ is also independent
of the choice of trivializations of $TS$.

According to the Preliminaries in Section \ref{S:Preliminaries} any stable framing
$\operatorname{ind}(S,\varphi)$ can be represented by a framed circle $S_0$ in $\RR^{n+1}$
such that
\[
  TS_0 \oplus \varepsilon^n \cong TS_0 \oplus\nu(S_0) \cong (T\RR^{n+1})|_{S_0} \cong \varepsilon^{n+1}
\]
recovers the stable framing of $(S,\varphi)$.
It follows that
\[
  \operatorname{ind}(S,\varphi) = \delta(S_0,\varphi_0),
\]
where $\delta$ is the invariant constructed by Pontryagin, \cite[Theorem 20]{MR0115178}. We will
use a different notation: Let us denote by $[A]$ the homotopy class constructed above from the
stable framing and by $\overline{[A]}$ the element $[A]+1 \in \Omega_1^{\textrm{fr}}(S^n)\cong
\ZZ_2$ where $1$ is the non-trivial element. Thus we proved

\begin{lem}\label{L:stable class as matrix}
  We identify $\pi_1\left( \mathrm{GL}^+(n+1)  \right)$ with $\Omega_1^{\textrm{fr}}$ by the unique
  isomorphism $\ZZ_2 \to \ZZ_2$. Then
  \[
    \overline{[A]} = \operatorname{ind}(S,\varphi).
  \]
\end{lem}
\section{Computation of $\pi^n(M)$}\label{S:computation of pinM}

We start this section to explain the group structure of $\pi^n(M)$. Let $j \colon S^{n} \vee S^n \to S^n \times S^n$ be the
inclusion of the ($2n-1$)-skeleton of $S^n\times S^n$ (endowed with the standard CW structure)
then, since $M$ is $n+1$-dimensional CW complex, the induced map $j_\# \colon [M,S^n\vee S^n] \to
[M,S^n\times S^n]$ is an isomorphism. For $f,g \in \pi^n(M)$ the group structure is defined by \[
  f + g := (\id_{S^n} \vee\id_{S^n})_\# \circ (j_\#)^{-1}(f\times g). \] This makes $\pi^n(M)$ to
an abelian group.

Now, let $f \colon M \to S^n$ be a differentiable map and $x_0 \in S^n$ a regular value. We orient $S^n$
by the normal vector field pointing outwards and the standard orientation of $\RR^{n+1}$.

Let $\Psi \colon \pi^n(M) \to H^n(M;\ZZ)$ be the map $\Psi([f]):=f^{\ast}\sigma$ where $\sigma \in
H^n(S^n;\ZZ)$ is a fixed generator.
We define the analogous degree map $\kappa \colon \pi^n(M) \to \pi_1^S$, where $\pi_1^S$ is the
first stable homotopy group of spheres, as follows: $\kappa$ is the composition of
\[
  \pi^n(M) \overset{\sim}{\longrightarrow} \Omega_1^{\textrm{fr}}(M) \overset{\kappa}{\longrightarrow} \Omega_1^{\textrm{fr}}
  \overset{\sim}{\longrightarrow} \pi_1^S.
\]
where the first and the last isomorphism is again induced by the Pontryagin-Thom isomorphism.

\begin{thm}\label{T:cohomotopy computation}
  Let $M$ be a closed $(n+1)$-dimensional spin manifold. Then
  \begin{enumerate}[label=(\alph*)]
    \item\label{T:cc_a} The generator of $\ker\Psi \cong \ZZ_2$ is given by the homotopy class of the map $\eta\circ  \omega \colon
      M\to S^{n+1}$, where $\eta$ represents a generator of $\pi_{n+1}(S^n)$ and $\omega \colon M \to S^{n+1}$
      is a map of odd degree. Thus $\ker\Psi \cong \pi_{n+1}(S^n)$.
    \item\label{T:cc_b} Identifying $\pi_1^S$ with $\pi_{n+1}(S^n)$ the degree map $\kappa \colon \pi^n(M) \to
      \pi_1^S$ splits the short exact sequence \eqref{eq:SES}. Thus we have
      \[
        \pi^n(M) \longrightarrow H^n(M;\ZZ) \oplus \pi_{n+1}(S^n),\quad
        [f] \mapsto (f^\ast\sigma,\kappa([f])).
      \]
      is an isomorphism of abelian groups.
  \end{enumerate}
\end{thm}
\begin{proof}
  Clearly we have $[\eta\circ \omega ] \in \ker\Psi$. For \ref{T:cc_a} it is enough to check that
  $\kappa([\eta\circ \omega])$ is the zero in $\pi_1^S$. We choose an odd degree map $\omega \colon
  M \to S^{n+1}$ as follows: Let
  $\{p_1,\ldots,p_l\}$ be the preimage of a regular value $y_0$ and choose open sets $U_1,\ldots,U_l
  \subset M$ as well as
  $V\subset S^{n+1}$ such that for all $i=1,\ldots,l$
  \begin{enumerate}[label=(\alph*)]
    \item $U_i$ and $V$ are contractible,
    \item $p_i \in U_i$ and $y_0 \in V$,
    \item there are charts $\psi_i\colon U_i \to \RR^{n+1}$, $\psi \colon V \to \RR^{n+1}$,
    \item $\omega_i :=\omega|_{U_i}$ is an orientation preserving diffeomorphism onto $V$.
  \end{enumerate}
  Since $p$ has odd degree, $l$ has to be an odd number (such maps exists e.g. using the
  Pontryagin-Thom construction).
  Furthermore let $x_0 \in S^n$ be a regular value of $\eta$ and $S_0 =\eta^{-1}(x_0)$. We may assume
  that $S_0$ is connected (e.g. see \cite[Theorem C]{MR0226651}) and $S_0 \subset V$. Let
  $\varphi_0$ be the framing of $\nu(S_0)$ induced by $\eta$, then $0 \neq [S_0,\varphi_0]\in
  \Omega_1^{\textrm{fr}}(S^{n+1})\cong \pi_{n+1}(S^n) \cong \ZZ_2$ and therefore by definition we have
  $\operatorname{ind}(S_0,\varphi_0) \neq 0$.

  Denote by $S_i := \omega_i^{-1}(S_0)$ and frame $\nu(S_i)$ by $\varphi_0$ and $d\omega_i$.
  Then $C= S_1\cup \ldots \cup S_l$ together with the framings $\varphi_i$ is a Pontryagin manifold
  for $\eta\circ\omega$ to the regular value $x_0$. Note that $w(S_i)=0$ for $i=1,\ldots,l$,
  since they are contained in a chart of $M$. By Proposition \ref{P: dependence on spin
  structure} this means that their indices do not depend on the spin structure of $M$.
  Clearly we deduce $\operatorname{ind}(S_i,\varphi_i)
  = \operatorname{ind}(S_0,\varphi_0)\neq 0$ for all $i=1,\ldots,l$ and from that we infer
  \[
    \kappa([\eta\circ\omega]) = \sum_{i=1}^{l} \operatorname{ind}(S_i,\varphi_i)
    =l \cdot \operatorname{ind}(S_0,\varphi_0) \neq 0
  \]
  since $l$ is odd, which proves \ref{T:cc_a}.

  Part \ref{T:cc_b} follows directly from part \ref{T:cc_a}.
\end{proof}

\begin{cor}\label{C:Cohomotopy simply connected}
Suppose $M$ is simply connected, then, up to homotopy, there are exactly two maps $M \to S^n$ and
  one of them is the constant map. The homotopy class of the non-trivial map is represented by
  $\eta\circ \omega \colon M \to S^n$, see Theorem \ref{T:cohomotopy computation}.
\end{cor}

Finally  we would like to show, that $\kappa$ is natural with respect to maps between manifolds which
preserve the spin structure
\begin{prop}\label{P:naturality}
Suppose $\Phi \colon M_1 \to M_2$ is a map between two closed and connected spin manifolds of dimension $(n+1)$.
We assume that the spin structure of $M_1$ coincides with the pull-back spin structure by $\Phi$ of
  $M_2$. Then for the natural homomorphism $\Phi^\# \colon \pi^n(M_2) \to \pi^n(M_1)$,
  $f\mapsto \Phi \circ f$ we have
  \[
    \kappa\left( \Phi^\#(f)  \right) = \deg_2 \Phi \cdot \kappa(f).
  \]
  where $\deg_2 \Phi$ is the mod $2$ degree of $\Phi$. Therefore using the isomorphism
\[
  \pi^n(M)
  \cong H^n(M;\ZZ) \oplus\pi_{n+1}(S^n)
\]
  we have
  \[
    \Phi^\# \colon \pi^n(M_2) \to \pi^n(M_1),\quad
    (\alpha,\nu)\mapsto (\Phi^\ast(\alpha),\deg_2 \Phi \cdot \nu)
  \]
\end{prop}

\begin{proof}
  First note that $\Phi^\#$ is well-defined on the homotopy class of $\Phi$.
  For $f \in \pi^n(M_2)$ there is a decomposition $f=f_\alpha + f_\nu$ with $\kappa(f_\alpha)=0$,
  $f_\alpha^\ast(\sigma)=\alpha$ and $\kappa(f_\nu)=\nu$ as well as $f_\nu^\ast(\sigma)=0$.

  Let us show first $\Phi^\#(f_\alpha)=f_{\Phi^\ast(\alpha)}$. Clearly we have
  $\Phi^\#(f_\alpha)(\sigma)=\Phi^\ast(\alpha)$ thus it remains to show
  $\kappa(\Phi^\#(f_\alpha))=0$. Let $C_2$ be the preimage of a regular value of $f_\alpha$
  with a normal framing $\varphi_0$ such that $\kappa([C_2,\varphi_0])=0$. Moreover we may choose
  $f_\alpha$ such that each framed circle of $(C_2,\varphi_0)$ has index $0$. Deform $\Phi$ to be
  transversal to $C_2$, thus $C_1:=\Phi^{-1}(C_2)$ is a closed $1$-dimensional submanifold of
  $M_1$. The normal bundle to $C_1$ is isomorphic to the pull back of the normal bundle of
  $C_2$ by $\Phi$. This induces a framing on $C_2$ such that every framed circle thereof has index
  $0$ (note that the spin structure of $M_1$ is the pulled back by $\Phi$ from $M_2$) which is also
  the framing induced by the map $f_\alpha\circ \Phi$. But this means
  $\kappa(\Phi^\#(f_\alpha))=0$.

  On the other hand we may assume a preimage of a regular point in $S^n$ under $f_\nu$
  is a contractible circle $S_2$ in $M_2$ with normal framing $\varphi$ such that the index
  of the framed circle $(S_2,\varphi)$ is $\nu \in \pi_{n+1}(S^n)$.  Then making again
  $\Phi$ transverse to $S_2$ we obtain a normally framed submanifold $(C_1,\varphi)$ such
  that the index of each framed circle in $C_1$ has index $\nu$. As in the proof of Theorem
  \ref{T:cohomotopy computation} the degree of $(C_1,\varphi)$ is just $\deg_2\Phi \cdot \nu$.
  Therefore $\Phi^\#(f_\nu) = f_{\deg_2\Phi \cdot \nu}$ and the proposition follows.
\end{proof}

\begin{cor}\label{C:counting formula}
    Let $f \colon M \to S^n$ and $x_0 \in S^n$ a regular value. Write $S_1 \cup \ldots \cup S_k
    = f^{-1}(x_0)$ such that $S_i$ is a connected component of $f^{-1}(x_0)$ and denote
    by $\varphi_i$ the induced framing from $f$. Then the number
  \[
    \# \{ i : \kappa([S_i,\varphi_i])\neq 0 \} \mod 2
  \]
  does not depend on $x_0$ and is a homotopy invariant.
\end{cor}
\section{Application to vector bundles}\label{S:Application to vector bundles}
In this section $\pi \colon E \to M$ should denote an oriented vector bundle of rank $n$ endowed with a spin
structure. Let $s \colon M \to M$ be a section. If not otherwise stated, we say $s$ is
\emph{transversal} if $s$ is transversal to the zero section $0_E$ of $E$. For a transversal
section $s$ the zero locus $C$ is a smooth $1$-dimensional closed submanifold of $M$. The
differential $ds \colon TM \to TE$ restricted to $\nu(C)$ is an isomorphism of the vector bundles
$\nu(C) \to E|_C$. Since $E$ possess a spin structure, by Lemma \ref{L:Framing on C} $E|_C$
has a framing and with $ds$ this endows $\nu(C)$ with the framing $\varphi$ of $E|_C$. Note that
the homology class $[C] \in H_1(M;\ZZ)$ is the Poincar\'{e} dual of the Euler class of $E$.

\begin{prop}\label{P:framed divisor does not depend on the section}
  The class $[C,\varphi] \in \Omega_1^{\textrm{fr}}(M)$ does not depend on the section $s$.
\end{prop}

\begin{proof}
Let $s' \colon M \to E$ be another transversal section and denote the corresponding normally framed
  zero locus by $(C',\varphi')$. Let $s^\ast \colon M \times I \to \textrm{pr}^\ast(E)$
  be a section of $\textrm{pr}^\ast(E) \to M \times I$ (where $\textrm{pr} \colon M \times I \to M$) such that
  $s^\ast|_{M \times 0} = s$ and $s^\ast|_{M \times 1} = s'$. There we may deform $s^\ast$
  to a section $\hat s$ which is transverse to the zero section of $\textrm{pr}^\ast(E) \to M\times
  I$ and agrees with $s$ and $s'$ on the boundary of $M \times I$. The zero locus of $\hat s$, call
  it $\Sigma \subset M \times I$ is a bordism between $C$ and $C'$ by construction. Moreover by
  Lemma \ref{L:Framing on C} $T(M\times I)|_\Sigma$ inherits a framing from the spin structure of
  $M$ as well as $\nu(\Sigma)$
  from $d\hat s$ and the spin structure of $\textrm{pr}^\ast(E)|_\Sigma$. Thus $\Sigma$ is
  a normally framed bordism between $(C,\varphi)$ and $(C',\varphi')$.
\end{proof}
%
\begin{dfn}\label{D:framed divisor}
 The bordism class $[C,\varphi] \in \Omega_1^{\textrm{fr}}(M)$ constructed above is called the
  \emph{framed divisor of $E\to M$}. Furthermore we define the \emph{degree $\kappa(E)$ of $E$}  as
  $\kappa([C,\varphi])$

\end{dfn}

For $[C,\varphi]\in \Omega_1^{\textrm{fr}}(M)$ we denoted by $w(C) \in H^{n}(M;\ZZ_2)$ the
Poincar\'{e} dual of the $\ZZ_2$ fundamental class $[C] \in H_1(M;\ZZ_2)$. If $[C,\varphi]$ is the
framed divisor of $E \to M$ then $w(C)$ is the $n$-th Stiefel-Whitney class $w_{n}(E)$
(since $w_{n}(E)$ is the Euler class $e(E)$ modulo $2$). Therefore if $w_{n}(E)=0$ then
the degree of $E$ does not depend on the spin structure (see Lemma \ref{L:Framing on
C} and Proposition \ref{P: dependence on spin structure}).

\begin{prop}\label{P: framed divisor does not depend on the spin structures}
  If $w_n(E)=0$ then the framed divisor is independent of the spin structures on $M$ and $E$.
\end{prop}
For the next theorem we will need a technical Lemma.  Let $D^{m}$ denote the closed unit ball in
$\RR^m$ and consider a smooth map $f \colon D^{n+k+1} \to \RR^{n+1}$. Assume that $0 \in\RR^{n+1}$
is a regular value for $f$ and $\Sigma_f^k := f^{-1}(0)$ does not intersect the boundary of
$D^{n+k+1}$. Denote by $\varphi_f$ the induced framing on $\nu(\Sigma_f^k)$. Since $\Sigma_f^k$ is
a submanifold of $\RR^{n+k+1}$ the trivialization $\varphi_f$ defines a stable tangential framing of
$\Sigma_f^k$ thus the pair $(\Sigma_f^k,\varphi_f)$ defines an element  in
$\Omega_k^{\textrm{fr}}$. On the other side, consider
\[
  g \colon S^{n+k}= \partial D^{n+k+1} \to S^n,\quad g(x):=\frac{f(x)}{|f(x)|}
\]
and choose a regular value $y \in S^{n}$. Denote by $(\Sigma_g^k,\varphi_g)$ the induced
stably framed manifold.
\begin{lem}\label{L:local index section}
  With the notation above we have that $(\Sigma_f^k,\varphi_f)$ and $(\Sigma_g^k,\varphi_g)$ are
  stably framed bordant, thus they define the same element in $\Omega_k^{\textrm{fr}}$.
\end{lem}

\begin{proof}
There is an $\varepsilon >0$ such that the closed ball $D_\varepsilon$ centered in $0 \in
  \RR^{n+1}$ with radius $\varepsilon$ contains only regular values of $f$. The preimage
  of $D_\varepsilon$ under $f$ is a disc bundle $D(\Sigma_f^k)$ of the normal bundle
  $\nu(\Sigma_f^k \hookrightarrow \RR^{n+k+1})$. Denote by $S(\Sigma_f^k)$ its sphere bundle.
  Then $f|_{S(\Sigma_f^k)}$ has image $S_\varepsilon  = \partial D_\varepsilon$. Thus for $y' \in
  S_\varepsilon$, $\Sigma_{y'} = \left( f|_{S(\Sigma_f^k)}  \right)^{-1}(y')$ lies completely in $S(\Sigma_f^k)$. Moreover the Pontryagin
  manifold $(\Sigma_{y'},\varphi_{y'})$ is framed bordant to $(\Sigma_f,\varphi_f)$. Thus we would
  like to show that $(\Sigma_{y'},\varphi_{y'})$ represents the same element in $\Omega_{k}^{\textrm{fr}}$
  as $(\Sigma_g,\varphi_g)$. Since the normal bundle of $S(\Sigma_f^k)$ is trivial the framing
  $\varphi_{y'}$ induces a framing $\varphi_{y'}'$ on $\nu(\Sigma_{y'}\hookrightarrow S(\Sigma_{f}^k))$ such that
  $(\Sigma_{y'},\varphi_{y'})$ is stably framed bordant to $(\Sigma_{y'},\varphi_{y'}')$. But the
  latter normally framed manifold is the Pontryagin manifold to the map $f|_{S(\Sigma_f^k)}
  \colon S(\Sigma_f^k) \to S_\varepsilon$ at the point $y ' \in S_\varepsilon$.

  Let $N$ be the complement of the interior of $D(\Sigma_f^k)$ in $D^{n+k+1}$.
  Then $N$ is a framed cobordism between $S^{n+k}=\partial D^{n+k+1}$ and $S(\Sigma_f^k)$.
  The restriction of the map
  \[
    F \colon N \to S^n,\quad F(x):=\frac{f(x)}{|f(x)|}
  \]
  to $S^{n+k}$ is equal to $g$ and $F$ restricted to $S(\Sigma^k)$ is equal to
  $\varepsilon^{-1}\hat f$. Hence $F$ defines a framed bordism between $(\Sigma_g^k,\varphi_g)$
  and $(\Sigma_{y'},\varphi'_{y'})$ which proves the lemma.
\end{proof}
\begin{thm}\label{T:now where vanishing section}
  Let $E \to M$ be an oriented vector bundle of rank $n$ with $w_2(E)=0$ over a closed spin manifold $M$ of dimension
  $n+1$. Then $E$ admits a nowhere vanishing section if and only if the Euler class is zero and
  $\kappa(E)=0$.
\end{thm}
\begin{proof}
Suppose there is a nowhere vanishing section of $E$ then clearly this section is transverse and
  has an empty framed divisor. Thus from Theorem \ref{T:BordismGroupSplits} we have that the Euler
  class must be zero and $\kappa(E)=0$.

  Assume now that $e(E)=0$ and $\kappa(E)=0$. Consider the fibration
  \[
    S^{n-1} \longrightarrow B \textrm{SO}(n-1) \longrightarrow B \textrm{SO}(n).
  \]
  where $B \textrm{SO}(k)$ denotes the classifying space to the special orthogonal group
  $\textrm{SO}(k)$. Consider the classifying map $g \colon M \to B \textrm{SO}(n)$ for $E \to M$.
  There exists a nowhere vanishing section if and only if there is a lift $\hat g\colon M \to
  B \textrm{SO}(n-1)$ of $g$ up to homotopy.

  First we put a CW-structure on $M$ (e.g. induced by a Morse function) then over the $(n-2)$-skeleton
  of $M$ there exists such a lift $\hat g$ of $g$. The obstruction to extend the lift over the
  $n$-skeleton lies in $H^n(M;\pi_{n-1}(S^{n-1}))=H^n(M;\ZZ)$ which is given by the Euler class $e(E)$.
  Since this is assumed to be zero $\hat g$ extends over the $n$-skeleton of $M$. The obstruction
  to extend $\hat g$ over the top cell of $M$ lies in $H^{n+1}(M;\pi_{n}(S^{n-1}))\cong
  \pi_{n}(S^{n-1})\cong \ZZ_2$. Let $e_{n+1}$ be the top cell of $M$ and $\psi \colon \partial
  e_{n+1}\cong S^{n} \to M$ the corresponding attaching map. The bundle $E|_{e_{n+1}}$ is
  canonical isomorphic to $e_{n+1} \times \RR^n$. Let $\sigma \colon M \to E$ be a section
  which has no zeroes over the $n$-skeleton of $M$ and which is transverse to the zero section of $E$.
  Then consider the map
  \[
    g \colon \partial e_{n+1}\cong S^n \to S^{n-1},\quad g(x):=\frac{\sigma\circ\psi(x)}{|\sigma\circ\psi(x)|}
  \]
  (where the norm is take with respect to a euclidean bundle metric on $E$). The homotopy
  class of $g$ in $\pi_{n}(S^{n-1})$ is the obstruction to extend a no where vanishing section over the
  $n$-skeleton to the $(n+1)$-skeleton of $M$. Since $\pi_n(S^{n-1})$ is isomorphic to the stable
  homotopy group $\pi_1^S$ we consider the homotopy class of $g$ as an element therein.

  From Lemma \ref{L:local index section} we infer that the $[g] \in \pi_1^S\cong \Omega_1^{\textrm{fr}}$ is equal to the framed
  divisor $\kappa(E)$ of $E$ defined by $\sigma$, thus $E$ admits a no where vanishing section in case
  $e(E)=0$ and $\kappa(E)=0$.
\end{proof}

\begin{exmpl}\label{E:Spheres}
As an application of our theory we will reprove the following fact due to Whitehead
  \cite{MR0005737} and Eckmann \cite{MR0007645}: The number of linear independent vector fields on $S^{4k+1}$ is equal to $1$
  (see also \cite{MR0133837} and in \cite{MR0215317}).

  Denote by $\langle \cdot , \cdot  \rangle $ the standard euclidean product in $\RR^{4k+2}$. The
  vector field
  \[
    v \colon \RR^{4k+2}\to \RR^{4k+2},\quad v(x_1,x_2,\ldots,x_{4k+2})=(-x_2,x_1,\ldots,-x_{4k+2},x_{4k+1})
  \]
  defines a nowhere vanishing vector field on $S^{4k+1}$ since $\langle v(x), x \rangle =0$ for $x \in S^{4k+1}$. Let
  $E$ the subbundle of $TS^{4k+1}$ orthogonal to the line bundle spanned by $v$. For any vector field
  on $S^{4k+1}$ which is in every point linear independent to $v$ there is a nowhere vanishing
  section of $E$ \footnote{For any pair on orthonormal vector fields $v_1,v_2$ of $S^{4k+1}$ one can choose
  a new pair of orthonormal vector fields which consists of $v$ and an section of $E$.}. Since the
  Euler class of $E$ vanishes, it suffices to show that $\kappa(E)$ is
  not zero by Theorem \ref{T:now where vanishing section} (note that the spin structures of
  $S^{4k+1}$ and that of $E$ are unique up to homotopy).

  Consider now the vector field
  \[
    w \colon \RR^{4k+2} \to \RR^{4k+2},\quad
    w(x) = (0,0,-x_5,x_6,x_3,-x_4,-x_9,x_{10},x_7,-x_8,\ldots)
  \]
  Since $\langle w(x), x \rangle  = \langle w(x), v(x) \rangle=0$ we have that $w$ is a section of
  $E$. Furthermore $w$ is transverse to the zero section of $E$ and the zero locus is given by
  \[
    S=\{(x_1,x_2,0,\ldots,0) \in S^{4k+1} : x_1^2+x_2^2 =1\}.
  \]
  In Example \ref{Ex:First examples} we saw that $TS^{4k+1}|_S$ inherits the standard framing from
  the spin structure. But the induced framing on $E|_S$ cannot be the standard framing. To see
  this assume it inherits the standard framing and let $\tau_1,\ldots,\tau_n$ be a trivialization of
  $E|_S$, then, since the spin structure on $E$ is induced by $TS^{4k+1}$ and $v$, the map $S \to
  \mathrm{SO}(4k+2)$, $x\mapsto (x,v(x),\tau_1(x),\ldots,\tau_n(x))$ has to be nullhomotopic cf.
  Example \ref{Ex:First examples} (note that $v|_S$ is tangent to $S$) which is a contradiction.
  Thus from Example \ref{Ex:First examples} we deduce that the index of the framed divisor is not
  zero, hence $\kappa(E)=1$ and therefore $E$ does not admit a nowhere vanishing section from
  Theorem \ref{T:now where vanishing section}.
\end{exmpl}

\begin{rem}\label{rem: computation of N}
In \cite[Theorem 1.6]{MR1972251} the authors show, that for any $n$-dimensional CW-complex of dimension $X$
and any $k$-dimensional integral cohomology class $a \in H^k(X;\ZZ)$ there exists an oriented
vector bundle over $X$ whose Euler class equals $2 \cdot N(n,k) \cdot a$.

  Suppose $\dim X = 2k+1$. By Steenrod's exact
  sequence \eqref{eq:SES} it follows that the Hurewicz map $\pi^n(X) \to H^n(X;\ZZ)$ is surjective.
  Then for every $a \in H^n(X;\ZZ)$ there is a map $f_a \in \pi^n(X)$ such
  that $f_a^\ast(\sigma)=a$, where $\sigma \in H^n(S;\ZZ)$ denotes the generator such that $2\sigma$
  equals to the Euler class of the tangent bundle $TS^n$ of $S^n$. Clearly the vector bundle
  $f_a^\ast(TS^n)$ has Euler class $2 \cdot a$ and therefore $N(2k,2k+1)=1$ in the notation of \cite{MR1972251}.
\end{rem}

Note that any vector bundle over $S^n$ for $n\neq 2,4,8$ has an Euler class divisible by $2$, cf.
\cite{MR102805, MR126282}.
In the cases $n=2,4,8$ there are real vector bundles whose Euler class is a generator of
$H^n(S^n;\ZZ)$, namely the associated bundles to the Hopf fibrations $S^{2n-1} \to S^n$.
We deduce

\begin{prop}\label{P: vector bundles in special cases}
  Suppose $n=4$ or $n=8$ and let $M$ be a $(n+1)$-dimensional closed spin manifold.
  Denote by $\mathrm{Vect}_n(M)$ the set oriented vector bundles over $M$ of rank $n$ up to
  isomorphism. Let $E_0 \to S^{n}$ denote the oriented rank $n$ vector bundle such that the
  Euler class of $E_0$ is a generator of $H^n(S^n;\ZZ)$. Then the map
  \[
    \pi^n(M) \to \mathrm{Vect}_n(M), \quad f\mapsto f^\ast(E_0)
  \]
  is injective.
\end{prop}
\begin{proof}
  We consider $f_1,f_2 \in \pi^n(M)$
  such that $E_1:=f_1^\ast(E_0)\cong f_2^\ast(E_0)=:E_2$ since they represent the Euler class the respective bundles.
  This implies $f_1^\ast(\sigma)=f_2^\ast(\sigma)$ for a generator in $H^n(S^n;\ZZ)$. Thus it remains
  to show that $\kappa(f_1) = \kappa(f_2)$. Let $x_i \in S^n$ be a regular value for $f_i$ for
  $i=1,2$. There is a section $\sigma_{0,i} \colon S^n \to E_0$ which
  is transverse to the zero section with an isolated zero in $x_i$
  (note that the Poincar\'e dual of $x_i$ in $S^n$ represents the Euler class of $E_0$. Therefore
  $\sigma_{0,i}$ can only exist since if the Euler class is a generator, since the index of
  transverse sections is always $\pm 1$). Then
  $\sigma_i:=f^\ast(\sigma_{0,i})$ is a transverse section of $E_i$. Note that from the
  Pontryagin-Thom construction we may assume that $f_i^{-1}(x_i)$ is connected, hence
  the zero locus of $\sigma_i$ coincides with $f_i^{-1}(x_i)$. Moreover the framed
  divisor of $E_i$ coincides with the degree of $f_i$ (cf. Definitions \ref{D:invariant} and
  \ref{D:framed divisor}). Since $E_1\cong E_2$ we have $\kappa(E_1) \cong \kappa(E_2)$ by
  construction of the framed divisor and Proposition \ref{P:framed divisor does not depend on the
  section}. From $f_1^{\ast}(\sigma) = f_2^{\ast}(\sigma)$ and
  $\kappa(f_1)=\kappa(E_1)=\kappa(E_2)=\kappa(f_2)$ it follows from Theorem \ref{T:cohomotopy
  computation} that $f_1$ is homotopic to $f_2$.
\end{proof}

\bibliographystyle{acm}
\bibliography{counting-formula}
\end{document}